\documentclass[11pt]{article}
\usepackage[english]{babel}

\usepackage 
			[margin=10pt,
			font=small,
			labelfont=bf,
			labelsep=endash, 
			justification=raggedright, 
			singlelinecheck=false] 
			{caption}

\usepackage{
	amsopn,
	amsmath,
	amsfonts,
	amssymb,
	cite,
	color,
	comment,
	dsfont,
	graphicx,
	latexsym,
	stmaryrd,
	setspace,
	tikz
	}

\usetikzlibrary{decorations.pathmorphing,matrix}

\specialcomment{percom}
	{\bgroup \it}{\egroup}
\excludecomment{percom}

\usepackage{amsthm}
	\newtheorem{theo}{Theorem}[section]

	\newtheorem{lem}[theo]{Lemma}

	\theoremstyle{definition}
	\newtheorem{defn}[theo]{Definition}
	\newtheorem*{remark}{Remark}
	
	\newtheoremstyle{proof}
		{3pt}
		{3pt}
		{}
		{}
		{\itshape}
		{:}
		{.5em}
		{}

\newcommand{\Z}{\ensuremath{\mathbb{Z}} }

\newcommand{\R}{\ensuremath{\mathbb{R}} }

\newcommand{\E}{\ensuremath{\mathbb{E}} }

\newcommand{\bfl}{\ensuremath{{\bf l}}}
\newcommand{\bfi}{\ensuremath{{\bf i}}}
\newcommand{\bfj}{\ensuremath{{\bf j}}}

\newcommand{\tr}{\text{tr}}

\renewcommand{\phi}{\varphi}

\newcommand{\wt}{\widetilde}

\newcommand{\wX}{\ensuremath{\raisebox{0pt}[1pt][1pt]{$\wt{[X]}$}_n^N}}
\newcommand{\wXpow}{\ensuremath{\raisebox{0pt}[1pt][1pt]{\scriptsize{$\wt{[X]}$}}_n^N}}
\newcommand{\XN}{\ensuremath{{[X]}_n^N}\ }

\newcommand{\PP}{\mathbb{P}}

\newcommand{\stab}{\stackrel{d_{st}}{\longrightarrow}}

\newcommand{\toop}{\stackrel{\PP}{\longrightarrow}}

\newcommand{\bee}{\begin{equation}}
\newcommand{\eee}{\end{equation}}
\newcommand{\bea}{\begin{eqnarray}}
\newcommand{\eea}{\end{eqnarray}}
\newcommand{\bean}{\begin{eqnarray*}}
\newcommand{\eean}{\end{eqnarray*}}

\begin{document}
\title{On spectral distribution of high dimensional covariation matrices}
\author{ Claudio Heinrich \thanks{Department of Mathematics, University of Aarhus,
Ny Munkegade 118, 8000 Aarhus C,
Denmark, Email: claudio.heinrich@math.au.dk.} \and
Mark Podolskij \thanks{Department of Mathematics, University of Aarhus,
Ny Munkegade 118, 8000 Aarhus C,
Denmark, Email: mpodolskij@math.au.dk.}}

\date{\today}

\maketitle

\begin{abstract}
In this paper we present the asymptotic theory for spectral distributions of high dimensional 
covariation matrices of Brownian diffusions. More specifically, we consider $N$-dimensional 
It\^o integrals with time varying matrix-valued integrands. We observe $n$
equidistant high frequency data points of the underlying Brownian diffusion and we assume that $N/n\rightarrow c\in (0,\infty)$.
We show that under a certain mixed spectral moment condition the spectral distribution of the empirical covariation matrix
converges in distribution almost surely. Our proof relies on method of moments and applications of graph theory.

\ \

{\it Keywords}:  diffusion processes, graphs, high frequency data, random matrices.\bigskip

{\it AMS 2010 Subject Classification:} 62M07, 60F05, 62E20, 60F17.

\end{abstract}

\section{Introduction}
\label{sec1}
\setcounter{equation}{0}
\renewcommand{\theequation}{\thesection.\arabic{equation}}

Last decades have witnessed an immense progress in the theory of random matrices and their applications
to probability, statistical physics and number theory. Since the seminal work \cite{W}, and increasingly so since \cite{MP}, the 
asymptotic behaviour of the spectrum of random matrices received a great deal of attention. We refer
to the monographs \cite{AGZ,BS,PS} for a detailed exposition of recent results and techniques.

This paper is devoted to the study of spectral distribution of empirical covariation matrices of Brownian integrals. On a filtered
probability space $(\Omega, \mathcal F, (\mathcal F_t)_{t\in [0,1]}, \mathbb P)$ we consider a diffusion process $(X_t)_{t\in [0,1]}$
that is defined as

\begin{align} \label{xdef}
X_t= X_0+\int_0^t f_s dW_s,
\end{align} 
where $W$ denotes an $N$-dimensional Brownian motion and $f$ is a $\R^{N \times N}$-valued step function given as
\begin{align} \label{fdef}
f_t = \sum_{l=1}^m T_l 1_{[t_{l-1}, t_l)} (t) 
\end{align} 
where $0=t_0< \cdots < t_m=1$ is a fixed partition of the interval $[0,1]$ and the matrices $T_j$, $1\leq j\leq m$, are either deterministic
or independent of the driving Brownian motion $W$. In mathematical finance one of the most central objects
is the empirical covariation of $X$, which is defined via
\begin{align} \label{vardef}
[X]^N_n := \sum_{i=1}^n \left(X_{\frac in} - X_{\frac{i-1}{n}} \right) \left(X_{\frac in} - X_{\frac{i-1}{n}} \right)^{*}.
\end{align}  
Here and throughout the paper $A^{^*}$ denotes the transpose of a matrix $A$. For a fixed dimension $N$ it is well known
that $[X]^N_n$ converges to the covariation matrix $[X]^N= \int_0^1 f_s f_s^{^*} ds$ as $n\rightarrow \infty$ whenever
the It\^o integral at \eqref{xdef} is well defined. When $N$ converges to infinity at the same rate as $n$ the situation becomes
much more delicate. In the following we briefly review some recent work on spectral distribution of large covariance/covariation matrices.
Recall that for a given matrix $A\in \R^{N \times N}$ with real eigenvalues $\lambda_1, \ldots, \lambda_N$ the spectral distribution 
of $A$ is defined via
\[
F^A (x):= \frac 1 N \sum_{j=1}^N 1_{\{\lambda_j\leq x\}}.
\]    
In \cite{EK} the author studies the spectral distribution of the empirical high dimensional covariance matrix based on i.i.d. data,
which corresponds to our model \eqref{xdef} with $f$ being constant. In this framework the spectral distribution of the empirical covariance
matrix converges and, more importantly, there is a one-to-one connection between the limit of the Stieltjes transform  of $F^{[X]^N_n}$
and the limit of $F^{[X]^N}$ (given the latter exists). It is exactly this relationship, called Mar\v{c}enko-Pastur equation, which makes
the estimation of the spectral distribution of the covariation matrix $[X]^N$ possible. In another paper \cite{ZL} the authors 
consider the model \eqref{xdef}, where the time variation of $f$ comes solely from a scalar function. In other words, they study 
processes of the type $f_s = a_s T$, where $a:\R \rightarrow \R$ is a scalar function and $T\in \R^{N \times N}$. In this situation the 
methods of \cite{EK} can not be directly applied to infer $F^{[X]^N}$, but a certain modification of the functional $[X]^N_n$, 
which separates the scalar function $a$ and the matrix $T$, still leads to a feasible procedure. 

Unfortunately, both methods do not work when the function $f$ has the form \eqref{fdef}. More precisely, the Stieltjes transform
method is hardly applicable in our setting unless all matrices $T_1, \cdots, T_m$ have the same eigenspaces for all $N$. In this work we follow
the route of method of moments, which has been originally proposed by \cite{YK} in the context of random matrices. The basic idea is to show
the almost sure convergence of all moments of the random probability measure $F^{[X]^N_n}$. Then, under Carleman's condition, the limiting 
distribution is uniquely determined by the limits of moments. The idea of the proof is heavily based on combinatorics of colored graphs.   
The main result of the paper is the following theorem.

\begin{theo} \label{th1}
Assume that $N/n\rightarrow c\in (0,\infty)$ and the following conditions hold: \newline
(i) There exists a constant $\tau_0>0$ such that $\|T_l\|_{\text{op}}\leq \tau_0$ for all $1\leq l\leq m$ and uniformly in $N$. \newline
(ii) For any $k\geq 1$ and any multi-index ${\bf l} \in \{1, \ldots, m\}^k$ the mixed spectral moment condition holds:
\begin{align} \label{msm}
M_{\bf{l}}^k := \lim_{N\rightarrow \infty} \frac 1N \text{\rm tr} \left( \prod_{i=1}^k T_{l_i}T^*_{l_i} \right) \quad \text{exists in the almost sure sense and is non-random}. 
\end{align} 
\begin{percom}
almost surely: The quantity $\frac 1N \text{\rm tr} \left( \prod_{i=1}^k T_{l_i}T^*_{l_i} \right)$ tends almost surely to a nonranom $M_{\bf{l}}^k$.
\end{percom}
Then $F^{[X]^N_n}$ converges in distribution to a non-random probability measure $F$ almost surely. The $k$-th moment $m_k$ of $F$ is given
via
\begin{align} \label{moments}
m_k=\sum_{r=1}^{k}c^{r-1}\sum_{\nu_1+...+\nu_r=k}\ \sum_{\bfl'\in\{1,...,m\}^{k}}\ 
 c_{r,\nu,\bfl'}\prod_{a=1}^r M^{\nu_a}_{\bfl^{(a)}}\prod_{l=1}^m (t_l-t_{l-1})^{s_{l,\nu,\bfl'}},
\end{align}
where $\bfl^{(a)}=(l^{(a)}_1,...,l^{(a)}_{\nu_a})\in\{1,...,m\}^{\nu_a}$ are such that $\bfl'=(\bfl^{(1)},...,\bfl^{(r)}).$ The power $s_{l,\nu,\bfl'}$ is defined as $s_{l,\nu,\bfl'}=\sum_{a=1}^r n_l^{(a)}$ where 
\[n_l^{(a)}=\begin{cases}\#\{j\: : \:l^{(1)}_j=l\} & \text{if $a=1$,}\\
\#\{j\neq 1\: : \: l^{(a)}_j=l\}& \text{else.}
\end{cases}\]
The definition of $c_{r,\nu,\bfl'}$ is given in section \ref{cdef}.
\end{theo}

The paper is structured as follows. In section \ref{sec2} we present an overview about related problems and
give some remarks on the conditions of Theorem \ref{th1}. At the end of this section we also give the definition of the constant $c_{r,\nu,\bfl'}$. Section \ref{sec3} is devoted to the proof of Theorem \ref{th1}.

\section{Related problems and remarks}
\label{sec2}
\setcounter{equation}{0}
\renewcommand{\theequation}{\thesection.\arabic{equation}}
In this section we review some related studies and comment on conditions of Theorem \ref{th1}.

\subsection{Limit theory for a fixed dimension $N$} 
As we mentioned in the introduction, the definition of a covariation matrix implies the convergence in probability
\[
[X]^N_n \toop [X]^N \qquad \text{as } n\rightarrow \infty 
\] 
when the dimension $N$ is fixed. Furthermore, the asymptotic results of \cite[Theorem 2.5]{BGJPS06} imply the following theorem.

\begin{theo} \label{th2}
Assume that the process $f$ is  c\`adl\`ag (not necessarily of the form \eqref{fdef}). Then we obtain the stable convergence
\begin{align} \label{cltfixedN}
\sqrt{n} \left( [X]^N_n - [X]^N \right) \stab \int_0^1 A_s^{1/2} dW'_s,
\end{align}
where $W'$ is a $N^2$-dimensional Brownian motion independent of the $\sigma$-algebra $\mathcal F$ and the $N^2 \times N^2$-dimensional matrix $A_s$
is given as
\[
A_s^{jk,j'k'}= C_s^{jj'} C_s^{kk'} + C_s^{jk'} C_s^{kj'} \qquad \text{with} \qquad C_s = f_sf_s^{*}.
\]  
\end{theo}
Quite surprisingly, Theorem \ref{th2} holds for general c\`adl\`ag stochastic processes $f$. We remark that Theorem \ref{th2}
can be transformed into a feasible standard central limit theorem (cf. \cite[Example 3.5]{PV}), thus making statistical inference
for components of $[X]^N$ possible. Such general results do not hold anymore when $N\rightarrow \infty $ and one requires much stronger conditions 
on the process $f$.

\subsection{Limit theory in the setting $N/n\rightarrow c\in (0,\infty)$}
In this subsection we shortly review the results of \cite{EK,ZL}. In \cite{EK} the author considers empirical covariance matrices of
i.i.d. vectors. In the setting of our model \eqref{xdef} it means that the function $f$ is deterministic and constant over the interval 
$[0,1]$. In order to state the main result we introduce the Stieltjes transform, which is defined via
\begin{align} \label{stielttransform}
m_{\mu} (z) = \int_{\R} \frac{1}{x-z} \mu(dx), \qquad z\in \mathbb{C}^+,
\end{align} 
where $\mu$ is a measure on $\R$ and $\mathbb{C}^+:= \{z\in \mathbb{C}:~\text{Im}~z>0\}$. Since the matrix $f$ is constant, we can write 
(in distribution)
\[
X_t= X_0+\Sigma^{1/2} W_t \qquad \text{with} \qquad \Sigma= [X]^N.  
\] 
The following path breaking result, called Mar\v{c}enko-Pastur equation, has been shown in \cite{MP} for the case of a diagonal 
matrix $\Sigma$ and extended later to general covariance matrices $\Sigma$ in \cite{S}.  

\begin{theo} \label{th3}
Assume that the spectral distribution $F^{\Sigma}$ of $\Sigma$ converges in distribution to $F$ as $N\rightarrow \infty$. When 
$N/n\rightarrow c\in (0,\infty)$ the following results hold: \newline
(i) Define the function $v_{[X]^N_n}(z):= -z^{-1}(1- N/n) + N m_{F^{[X]^N_n}} (z)/n$ for $z\in \mathbb C^{+}$. 
Then there exists a deterministic function $v$ such that
\[
v_{[X]^N_n}(z) \rightarrow v(z) \quad \text{almost surely.}
\]
(ii) The function $v$ from (i) satisfies the Mar\v{c}enko-Pastur equation
\begin{align} \label{mpequation}
-\frac{1}{v(z)} = z-c \int_{0}^{\infty} \frac{x}{1+xv(z)} F(dz).
\end{align}
(iii) The equation \eqref{mpequation} has a unique solution, which is the Stieltjes transform of a measure.
\end{theo}

In practice it is of course impossible to check whether the spectral distribution $F^{\Sigma}$ converges as $N\rightarrow \infty$. 
A pragmatic solution to this problem is to assume that $N$ is large enough, so that $F^{\Sigma}$ can be identified with its theoretical
limit $F$. In the next step, as proposed in \cite{EK}, discretization and convex optimization can be applied to construct a numerical
algorithm to compute the function $F$ from Mar\v{c}enko-Pastur equation \eqref{mpequation}. At this step the approximation $v_{[X]^N_n}(z)\approx v(z)$
can be used. Finally, since we have identified  $F^{\Sigma}$ with $F$, the spectral density of the covariance matrix $\Sigma$ can be 
recovered from $F$. This procedure shows the importance of Mar\v{c}enko-Pastur equation for statistical inference.

In the work \cite{ZL} the authors propose an extension of this procedure to time-varying matrices $f_s$, where the time variation is described
by a scalar function. More precisely, they consider models of the type \eqref{xdef} with 
\[
f_s = a_s \Sigma^{1/2},
\] 
where $a: [0,1] \rightarrow \R$ is a scalar function and $\Sigma$ is a positive definite matrix with $\text{tr}(\Sigma)=N$ (possibly random, but independent of $W$).
In this setting the Mar\v{c}enko-Pastur law for $[X]^N_n$ can not be expected to hold in general as it has been demonstrated in \cite[Proposition 3]{ZL}.
The functional $[X]^N_n$ requires a modification to satisfy the Mar\v{c}enko-Pastur equation \eqref{mpequation}. Such a modification is given as
\begin{align} \label{hatx}
\raisebox{0pt}[1pt][1pt]{$\widehat{[X]}$}_n^N:= \frac{\text{tr}([X]^N_n)}{n} \sum_{i=1}^n  
\frac{\left(X_{\frac in} - X_{\frac{i-1}{n}} \right) \left(X_{\frac in} - X_{\frac{i-1}{n}} \right)^{*}}
{|X_{\frac in} - X_{\frac{i-1}{n}}|^2},
\end{align}   
where $|\cdot|$ denotes the Euclidean norm. Intuitively speaking, the proposed transformation of the original statistic $[X]^N_n$ eliminates the scalar
variation $a_s$ and the methods of \cite{EK} become applicable. Indeed, under certain conditions, the spectral distribution $F^{\widehat{[X]}_n^N}$
is connected to $F^\Sigma$ through the Mar\v{c}enko-Pastur equation \eqref{mpequation}. We refer to \cite[Theorem 2]{ZL} for a detailed exposition
of the asymptotic theory.

\subsection{Remarks on conditions of Theorem \ref{th1}}
\label{miscdisc}

In this subsection we provide a discussion of conditions of Theorem \ref{th1}. 

First of all, we remark that the mixed spectral moment condition at \eqref{msm} is a rather strong condition, which however seems 
to be necessary according to our proofs. Nevertheless, in some special cases this assumption can be replaced by an easier condition.
For instance, in the setting of a constant function $f$, i.e. $T_1=\ldots=T_m=T$, a necessary condition for Theorem \ref{th1} to hold
becomes
\begin{align} \label{altcond}
F^{TT^*} \longrightarrow F,
\end{align}
where $F^{TT^*}$ is the spectral distribution of $TT^*$ and the convergence is in distribution almost surely towards a non-random distribution function
$F$. This assumption is used in classical works \cite{SB,YK}. In this framework the boundedness of the operator norm at (i) of Theorem \ref{th1}
is not required as this condition can be overcome by a truncation argument. More precisely, defining $F^{TT^*}_{\tau} (x):= N^{-1} \sum_{i=1}^N
1_{\{\lambda_i\leq x\wedge \tau\}}$, assumption \eqref{altcond} implies the convergence 
\[
F^{TT^*}_{\tau} \longrightarrow F_{\tau},
\]  
where $F_{\tau}$ is a non-random distribution function, for all $\tau >0$. The convergence of moments result similar to \eqref{moments} is then
proved by showing the corresponding assertion for a fixed $\tau$ and letting $\tau \rightarrow \infty$. We refer to e.g. \cite{BS} for a detailed
exposition. Also the condition \eqref{msm} of  
Theorem \ref{th1} follows directly from \eqref{altcond} and boundedness of $\|TT^*\|_{op}$ due to the obvious relation 
\[
\frac 1N \text{\rm tr} \left( TT^* \right)^k = \int x^k F^{TT^*} (dx).
\]  
However, in the general framework of \eqref{fdef} the convergence of, say, joint spectral distribution of matrices $T_1T_1^{*}, \ldots, T_mT_m^{*}$ 
is not sufficient
to conclude convergence of mixed spectral moments at \eqref{msm}. The reason is that the behaviour of the expression at \eqref{msm}
is not solely determined by the eigenvalues of the involved matrices, but crucially depends on their eigenspaces. For the very same reason 
the truncation argument of \cite{YK} does not work, and spectral boundedness at (i) of Theorem \ref{th1} has to be assumed explicitly. Therefore
it seems hard to avoid imposing condition \eqref{msm}. Let us remark however that when matrices $T_1T_1^{*}, \ldots, T_mT_m^{*}$ have the same eigenspaces
for all $N$, i.e. 
there exist eigenvectors $v_1, \ldots v_N$ such that $T_lT_l^{*} v_i= \lambda_i^{(l)}v_i$, 
then conditions (i) and (ii) of Theorem \ref{th1} can be replaced by assuming the almost sure weak convergence of the joint spectral distribution       
\[
F^{(T_1,...,T_m)}(x_1,...,x_m)=\frac 1N \sum_{i=1}^N 1_{\{\lambda_i^{(1)}\leq x_1,...,\lambda_i^{(m)}\leq x_m\}}
\]
towards a non-random distribution function $F$.

It is worth noticing that in the framework of free probability the mixed moment condition is referred to as the convergence of the joint distribution of the noncommutative random variables $T_1T_1^*,...,T_mT_m^*$, as $N\to \infty.$ See \cite{Bi,VDN} for an overview of this theory and its applications to random matrix theory. 
In particular, asymptotic freeness of $T_1T_1^*,...,T_mT_m^*$ allows to weaken the mixed moment condition. Denoting for $N\times N$ random matrices $\tau_N(A)=\frac 1N \E[\tr(A)]$, the matrices $T_1T_1^*,...,T_mT_m^*$ are asymptotically free if for all $i_1\neq i_2\neq...\neq i_k$ and all $p_1,...,p_k>0$ 
\[\lim_{N\to\infty}\tau_N \left[\left( (T_{i_1}T_{i_1}^*)^{p_1}-\tau_N \left( (T_{i_1}T_{i_1}^*)^{p_1}\right)\right)\cdots \left( (T_{i_k}T_{i_k}^*)^{p_k}-\tau_N \left((T_{i_1}T_{i_1}^*)^{p_1}\right)\right)\right]=0.\]
By linearity of $\tau_N$ it is then obvious that all mixed limiting moments exist if and only if the spectral distributions $F^{T_iT_i^*}$ converge to nonrandom limiting distributions $F_i$ with finite moments of all orders for $i=1,...,m$, almost surely.

\subsection{Definition of $c_{r,\nu,\bfl'}$}
\label{cdef}
In this subsection we give the definition of the constant $c_{r,\nu,\bfl'}$ that appears in Theorem \ref{th1}.

Given $\bfl'\in\{1,...,m\}^k$ and $\nu_1,...,\nu_r$ with $\nu_1+\dots+\nu_r=k$, we let $\bfl^{(a)}=(l^{(a)}_1,...,l^{(a)}_{\nu_a})\in\{1,...,m\}^{\nu_a}$ such that $\bfl'=(\bfl^{(1)},...,\bfl^{(r)}).$ We recall the definition 
\[n_l^{(a)}=\begin{cases}\#\{j\: : \:l^{(1)}_j=l\} & \text{if $a=1$,}\\
\#\{j\neq 1\: : \: l^{(a)}_j=l\}& \text{else.}
\end{cases}\]
Given a tree, i.e. a connected graph without cycles, $G$ with $r$ vertices $H_1,...,H_r$, we define for $l\in\{1,...,m\}$ and $a\in\{1,...,r\}$ numbers $n_l^{(a),G}$ in the following way: Let $H_{a_1},...,H_{a_p}$ be the vertices adjacent to $H_a$ in $G$ (i.e. the vertices connected to $H_a$ by a path of length 1), where we leave out the vertex that lies on the path from $H_a$ to $H_1$, if $a>1$. We set 
\[n_l^{(a),G}= \#\{j\in\{1,...,p\}\: :\: l^{(a_j)}_1=l\}.\]
Then, we have
\begin{align*}
c_{r,\nu,\bfl'}&= \sum_{G}\frac{1}{|S_{\bfl',G}|}\prod_{l=1}^m \prod_{a=1}^{r}\frac{n_l^{(a)}!}{(n_l^{(a)}-n_l^{(a),G})!} 1_{\{n_l^{(a),G}\leq n_l^{(a)}\}},
\end{align*}
where the summation runs for all trees $G$ on $\{H_1,...,H_r\}$. Here, $S_{\bfl',G}$ is the set of all permutations $\pi$ on the $\{2,...,r\}$ for which at least one of the following holds:
\begin{itemize}
\item [{\it (i)}] $\bfl^{(\pi(p))}\neq\bfl^{(p)}$ for some $p\in\{2,...,r\}$ 
\item [{\it (ii)}] $G_\pi\neq G$, where $G_\pi$ is the graph obtained from $G$ by permuting the vertices $H_2,...,H_r$ according to $\pi.$
\end{itemize}

\section{Proof}
\label{sec3}
\setcounter{equation}{0}
\renewcommand{\theequation}{\thesection.\arabic{equation}}

For the proof of Theorem \ref{th1} we rely on the well known moment convergence theorem.

\begin{theo}\label{MCT}
Let $(F_n)$ be a sequence of p.d.f.s with finite moments of all orders $m_{k,n}=\int x^k dF_n(x)$. 
Assume $m_{k,n}\to m_k$ for $n\to\infty$ for $k=1,...$ where 
\begin{itemize}
\item[(a)] $m_k<\infty$ for all $k$ and 
\item [(b)]$\sum_{k=1}^\infty [m_{2k}(F)]^{-\frac 1 {2k}}=\infty.$
\end{itemize}
Then, $F_n$ converges weakly to the uniquely determined probability distribution function $F$ with moment sequence $(m_k).$
\end{theo}
Condition (b) is known as {\it Carleman's condition.} For the proof we refer to \cite[Theorem 3.3.11]{D}.

 The strategy for proving Theorem \ref{th1} is the following:
In the next subsection we introduce colored $Q^+$-graphs. In the two subsections thereafter, these graphs take a key role in showing that
\begin{align}\label{expmklim}
\E[m_k(F^{[X]^N_n})]\to m_k
\end{align}
holds for all $k$, where $m_k$ is defined as in Theorem \ref{th1}.

Then, in subsection \ref{m_kcon} we argue that
\begin{align}\label{mklim}
 \E\left[\left(m_k(F^{[X]^N_n})-\E[m_k(F^{[X]^N_n})]\right)^4\right] = O(N^{-2}),
\end{align}
which yields $m_k(F^{[X]^N_n})\to m_k$, almost surely, by virtue of the Borel-Cantelli Lemma. Finally, verifying that the sequence $(m_k)$ satisfies Carleman's condition and applying Theorem \ref{MCT} completes the proof.

Our proof extends the proof given in \cite{YK} (see also \cite{BS} and \cite{B}) for the case of constant function $f$.
In order to deal with our more general setting we introduce colored graphs and use new combinatorical arguments.

Throughout the proof, we denote the entries of the matrices $T_l$ by $(T_l)_{ij}=t^{(l)}_{ij}$, and likewise for other matrices, in order to maintain readability.


\subsection{Colored $Q^+$-graphs}


For $l=1,...,m$ let $Y_l$ be $N\times [n(t_l-t_{l-1})]$ matrices containing i.i.d. standard normal variables independent of $T_l$, where $[n(t_l-t_{l-1})]$ denotes the integer part of $n(t_l-t_{l-1})$.
Set
\[\wX:= \frac 1 n \sum_{l=1}^m T_l Y_l Y_l^*T_l^*.\]
By virtue of the well known fact 
\[\|F^A-F^B\|_\infty\leq \frac 1 N \text{rank}(A-B)\]
for $N\times N$ symmetric matrices $A$ and $B$, it is easy to see that
\begin{align}\label{error}
\|F^{[X]^N_n}-F^{\wXpow}\|_\infty \to 0,
\end{align}
as $n,N\to\infty$. Therefore, we can  replace $\XN$ by $\wX$ for the proof of Theorem \ref{th1}.
Conditioning on all $T_l$ as given allows us, moreover, to restrict ourselves to nonrandom $T_l$ for the proof.

In order to show the convergence of the expected $k$-th spectral moment $\E[m_k(F^{\wXpow})]$ we are faced with the equation
\begin{align}\label{expmk}
\E[m_k(F^{\wXpow})]	&=\frac 1 {N}\frac 1 {n^k} \E\left[\tr\left(\sum_{l=1}^m T_lY_lY_l^*T_{l}^*\right)^k\right]
						\nonumber\\
						&= 	N^{-1} n^{-k} \E\left[\sum_{{\bf l},{\bf i},{\bf j}} t^{(l_1)}_{i_1 i_2}y^{(l_1)}_{i_2 j_1}
						y^{*(l_1)}_{ j_1 i_3}t^{*(l_1)}_{i_3 i_4}\cdots t^{(l_k)}_{i_{3k-2} i_{3k-1}}y^{(l_k)}_{i_{3k-1} j_k} 
						y^{*(l_k)}_{ j_k i_{3k}}t^{*(l_k)}_{i_{3k} i_{1}}\right],\nonumber\\
 \end{align}
Here, the summation runs over all $\bfl=(l_1,...,l_k)\in \{1,...,m\}^k$ and $\bfi=(i_1,...,i_{3k})\in \{1,...,N\}^{3k}$. For all $a$, the index $j_a$ varies over $\{1,...,[n(t_{l_a}-t_{l_a-1})]\}$.

In order to carry out the summation we introduce colored $Q^+$-graphs which correspond to the summands in the above equation.
These graphs are related to $Q$-graphs as used by the authors of \cite{YK}.

\begin{figure}
 \begin{tikzpicture}[>=latex,scale=0.6]
 \def\horizontaledge{[->] parabola bend +(1,0.15) +(2,0)}
 
 \draw (-12,2)--(12,2) node[right]{\bf i};
 \draw (-12,-2)--(12,-2) node[right]{${\bfj}^{(1)}$};
 \draw (-12,-3)--(12,-3) node[right]{${\bfj}^{(2)}$};

 
 \draw [green!70!black](-11,2) \horizontaledge;
 \draw[green!70!black,->] (-9,2) parabola bend +(-1,0.5) ++(-2,0);

 \draw [green!70!black](-7,2) \horizontaledge;
 \draw[red]   (-5,2) \horizontaledge;
 \draw[red]   (-3,2) \horizontaledge;
 \draw[green!70!black,->] (-1,2) parabola bend +(-3,1) +(-6,0); 
 
 \draw[red]   (1,2) \horizontaledge;
 \draw[red]   (3,2) \horizontaledge;
 \draw[red]   (5,2) \horizontaledge;
 \draw[green!70!black]   (7,2) \horizontaledge;
 \draw[green!70!black]   (9,2) \horizontaledge;
 \draw[red,->] (11,2) parabola bend +(-5,1.5) +(-10,0);


 \draw[green!70!black,->] (-9,2) ..controls (-8.75,0) .. (-8,-2);
 \draw[green!70!black,->] (-8,-2) ..controls (-8.25,0) .. (-9,2);
 \draw[green!70!black,->] (-8,-2) ..controls(-7.75,0).. (-7,2);
 \draw[green!70!black,->] (-7,2) ..controls(-7.25,0).. (-8,-2);
 
 \draw[red,->] (-3,2) ..controls +(0.5,-2.5)..+(2,-5);
 \draw[red,->] (-1,-3) ..controls +(-0.5,2.5).. +(-2,5);
 \draw[red,->] (1,2) ..controls +(-0.5,-2.5)..+(-2,-5);
 \draw[red,->] (-1,-3) ..controls +(0.5,2.5).. +(2,5);
 
 \draw[red,->] (5,2) ..controls +(-0.25,-2.5).. +(0,-5);
 \draw[red,->] (5,-3) ..controls +(0.25,2.5).. +(0,5);
 
 \draw[green!70!black,->] (9,2) ..controls +(-0.25,-2).. +(0,-4);
 \draw[green!70!black,->] (9,-2) ..controls +(0.25,2).. +(0,4);


\draw (-11,2) node [below] {\tiny $i_1=i_{19}$};
\draw (-9.4,2) node [above right] {\tiny $i_2=i_{18}$};
\draw (-7.2,2) node [below right] {\tiny $i_3=i_{17}$};
\draw (-5,2) node [above] {\tiny $i_4$};
\draw (-3,2) node [above] {\tiny $i_5=i_{15}$};
\draw (-1,2) node [below] {\tiny $i_{16}$};
\draw (1.4,2) node [above left] {\tiny $i_6=i_{14}$};
\draw (3,2) node [below] {\tiny $i_7$};
\draw (5,2) node [above] {\tiny $i_8=i_9$};
\draw (7,2) node [below] {\tiny $i_{10}$};
\draw (9,2) node [above] {\tiny $i_{11}=i_{12}$};
\draw (11,2) node [below] {\tiny $i_{13}$};


\draw (-8,-2) node [below] {\tiny $j_1=j_{6}$};
\draw (-1,-3) node [below] {\tiny $j_2=j_5$};
\draw (5,-3) node [below] {\tiny $j_3$};
\draw (9,-2) node [below] {\tiny $j_4$};

\end{tikzpicture}
\caption{A colored $Q^+$-graph for $k=6$ and $m=2$. Here, $\bfl=(1,2,2,1,2,1)$ where $1=$ green and $2=$ red. }
\label{Q+ex}
\end{figure}
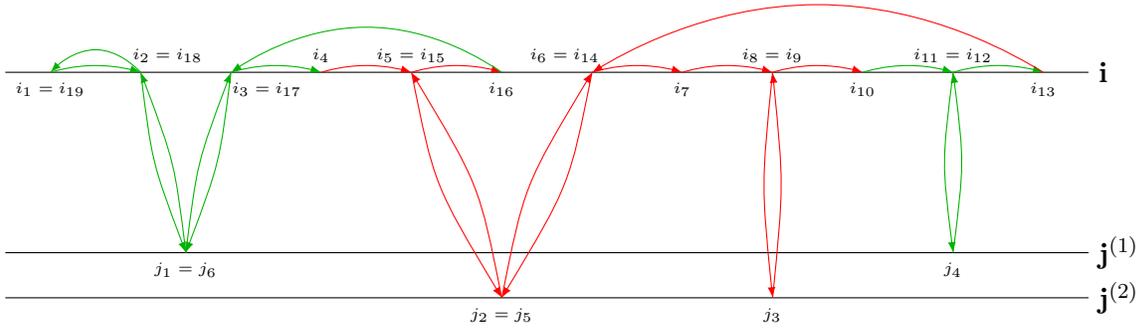

\begin{defn}
Let $k>0$. Given multi-indices $\bfl\in\{1,\dots,m\}^k$, ${\bf i}\in \{1,...,N\}^{3k}$, and ${\bf j}=(j_1,...,j_k)$ where $j_a\in \{1,\dots,[n(t_{l_a}-t_{l_a-1})]\}$, we define the {\it colored $Q^+$-graph} $Q_{\bfl,\bfi,\bfj}$ in the following way.
Choose $m$ arbitrary colors $c_1,...,c_m$. For brevity we will usually not distinguish between $l\in\{1,...,m\}$ and its associated color $c_l.$
Draw $m+1$ horizontal lines, the $\bfi$-, $\bfj^{(1)}$-, \dots, $\bfj^{(m)}$-line. Mark the numbers $\{1,...,N\}$ on the $\bfi$-line and, for all $l,$ the numbers $\{1,...,[n(t_l-t_{l-1})]\}$ on the $\bfj^{(l)}$-line.  For $s=1,...,k,$ draw horizontal edges colored in $l_s$ from $i_{3s-2}$ to $i_{3s-1}$ and from $i_{3s}$ to $i_{3s+1}$ with the convention that $i_{3k+1}= i_1$. For $s=1,...,k$, draw a vertical (down) edge from $i_{3s-1}$ to $j_{s}$ on the $\bfj^{(l_s)}$-line and a vertical (up) edge from $j_{s}$ to $i_{3s}$, both edges also colored in $l_s$.
 The result is a connected directed graph forming a cycle. It consists of $4k$ edges and always 4 subsequent edges are of the same color.
Figure \ref{Q+ex} provides an example of a colored $Q^+$-graph.
\end{defn}

There is a one to one correspondence between colored $Q^+$-graphs and the summands of (\ref{expmk}). Highlighting this correspondence we introduce the notation 
\begin{align}\label{tydef}
(ty)_{Q_{\bfl,\bfi,\bfj}}= \E\left[ t^{(l_1)}_{i_1 i_2}y^{(l_1)}_{i_2 j_1}
						y^{*(l_1)}_{j_1 i_3}t^{*(l_1)}_{i_3 i_4}\cdots t^{(l_k)}_{i_{3k-2} i_{3k-1}}y^{(l_k)}_{i_{3k-1} j_k} 
						y^{*(l_k)}_{j_k i_{3k} }t^{*(l_k)}_{i_{3k} i_{1}}\right].
\end{align}

We will organize the colored $Q^+$-graphs in three categories and then derive the limit for (\ref{expmk}) if the summation runs only for graphs from one of these categories. To this end, the following definitions are required.

\begin{defn}
The {\it head} $H(Q_{\bfl,\bfi,\bfj})$ of a colored $Q^+$-graph $Q_{\bfl,\bfi,\bfj}$ is the subgraph of all vertices on the \bfi-line and all horizontal edges.
\end{defn}

\begin{defn}
The {\it pillar} of a colored $Q^+$-graph $Q_{\bfl,\bfi,\bfj}$ is the Graph obtained from $Q_{\bfl,\bfi,\bfj}$ by first gluing together coincident vertical edges, then gluing all vertices on the \bfi-line that are connected in the head of $Q_{\bfl,\bfi,\bfj}$, and then deleting all horizontal edges. The pillar is undirected and colorless.
See Figure \ref{pilex} for an example.
\end{defn}

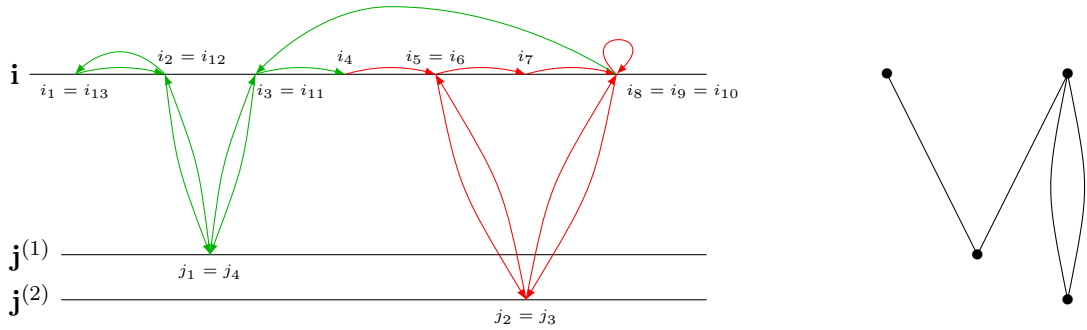
\begin{figure}
 \begin{tikzpicture}[>=latex,scale=0.6]
 \def\horizontaledge{[->] parabola bend +(1,0.15) +(2,0)}
 
  \draw (0,2)node[left]{\bf i}--(15,2) ;
 \draw (0.7,-2)node[left]{$\bfj^{(1)}$}--(15,-2);
 \draw (0.7,-3)node[left]{$\bfj^{(2)}$}--(15,-3);
 
 \draw [green!70!black](1,2) \horizontaledge;
 \draw [green!70!black,->] (3,2) parabola bend +(-1,0.5) ++(-2,0);

 \draw [green!70!black](5,2) \horizontaledge;
 \draw[red]   (7,2) \horizontaledge;
 \draw[red]   (9,2) \horizontaledge;
 \draw [red](11,2) \horizontaledge;
 \draw[red,->] (13,2) ..controls +(-1,1) and +(1,1).. (13,2); 
  \draw[green!70!black,->] (13,2) parabola bend +(-5,1.5) +(-8,0);

  \draw[green!70!black,->] (3,2) ..controls +(0.25,-2) .. +(1,-4);
  \draw[green!70!black,->] (4,-2) ..controls +(-0.25,2) .. +(-1,4);
  \draw[green!70!black,->] (5,2) ..controls +(-0.25,-2) .. +(-1,-4);
  \draw[green!70!black,->] (4,-2) ..controls +(0.25,2) .. +(1,4);

  \draw[red,->] (9,2) ..controls +(0.5,-2.5) .. +(2,-5);
  \draw[red,->] (11,-3) ..controls +(-0.5,2.5) .. +(-2,5);
  \draw[red,->] (13,2) ..controls +(-0.5,-2.5) .. +(-2,-5);
  \draw[red,->] (11,-3) ..controls +(0.5,2.5) .. +(2,5);

 \draw (1,2) node [below] {\tiny $i_1=i_{13}$};
 \draw (2.6,2) node [above right] {\tiny $i_2=i_{12}$};
 \draw (4.8,2) node [below right] {\tiny $i_3=i_{11}$};
 \draw (7,2) node [above] {\tiny $i_4$};
 \draw (9,2) node [above] {\tiny $i_5=i_{6}$};
 \draw (11,2) node [above] {\tiny $i_7$};
\draw (13,2) node [below right] {\tiny $i_8=i_9=i_{10}$};


 \draw (4,-2) node [below] {\tiny $j_1=j_4$};
 \draw (11,-3) node [below] {\tiny $j_2=j_3$};


  \def\vertex{ node{\small $\bullet$}}

\draw (19,2)\vertex;
\draw (19,2)  -- ++(2,-4) \vertex;
\draw (21,-2)-- ++(2,4) \vertex;
\draw (23,2)..controls +(-0.5,-2.5)..++(0,-5) \vertex;
\draw (23,-3)..controls +(0.5,2.5)..+(0,5);

 \end{tikzpicture}
\caption{A colored $Q^+$-graph in $\mathcal C_3$ and its pillar.}
\label{pilex}
\end{figure}

We divide the colored $Q^+$-graphs in the following three categories.
 Category $\mathcal C_1$ contains all graphs $Q_{\bfl,\bfi,\bfj}$ such that every down edge of $Q_{\bfl,\bfi,\bfj}$ coincides with exactly one up edge, and such that the pillar of $Q_{\bfl,\bfi,\bfj}$ is a tree. An example of a colored $Q^+$-graph in this category is the graph in Figure \ref{Q+ex}. Note that coincident vertical edges are always of the same color.
Category $\mathcal C_2$ contains all colored $Q^+$-graphs that have at least one single vertical edge.
Category $\mathcal C_3$ contains all other colored $Q^+$-graphs. The graph in Figure \ref{pilex} is in this category since its pillar contains a cycle. 

Now we can split the sum (\ref{expmk}) into
\begin{align}\label{catsum}
\E[m_k(F^{\wXpow})]=N^{-1} n^{-k}\left[\sum_{Q\in \mathcal C_1}(ty)_{Q}+\sum_{Q\in \mathcal C_2}(ty)_{Q}+\sum_{Q\in \mathcal C_3}(ty)_{Q}\right].
\end{align}

The second sum vanishes since a vertical edge in $Q$ which is single corresponds to a factor $y^{(l)}_{ij}$ in (\ref{tydef}) that occurs only once. Hence, the expectation is $0$ due to independence.

In the following section we argue that the third sum is negligible for $n,N\to\infty.$ In the section thereafter, the limit for the first sum is derived.


\subsection{The limit for the sum of $\mathcal C_3$ graphs}


We make the following conventions on notation: 
For a $Q^+$-graph $Q$ we denote by $r$ the number of connected components of the head. By $s_1,...,s_m$ we denote the numbers of noncoincident $\bfj^{(1)},...,\bfj^{(m)}$-vertices, respectively, and $s=s_1+\dots+s_m$. Denote further by $p$ the number of noncoincident vertical edges of $Q$. 

\begin{lem}\label{c3prop}
For a category $\mathcal C_3$ colored $Q^+$-graph $Q_{\bfl,\bfi,\bfj}$ it holds that $p+s-1 <k$. Furthermore, the degree of each vertex of $H(Q_{\bfl,\bfi,\bfj})$ is at least 2.
\end{lem} 
\begin{proof}
The pillar of $Q_{\bfl,\bfi,\bfj}$ has $r+s$ vertices and $p$ edges and is connected. Therefore, $r+s-1\leq p$ where equality implies that the pillar is a tree. 
We distinguish two different cases.

{\it Case 1.} If every vertical edge of $Q_{\bfl,\bfi,\bfj}$ has coincidence multiplicity 2, it holds that $p=k$, since $Q_{\bfl,\bfi,\bfj}$ contains $2k$ vertical edges. If, in this case, the pillar would be a tree, $Q_{\bfl,\bfi,\bfj}$ would be in $\mathcal C_1$. Therefore, we have $r+s-1<p=k$.

{\it Case 2.} One vertical edge of $Q_{\bfl,\bfi,\bfj}$ has coincidence multiplicity larger 2. We have $p<k$ and thus $r+s-1\leq p<k$.

Every \bfi-vertex of $Q_{\bfl,\bfi,\bfj}$ connects either with at least two horizontal edges or with one horizontal and one vertical edge, which is then single. Therefore, if some vertex of $H(Q_{\bfl,\bfi,\bfj})$ has degree one, we have $Q_{\bfl,\bfi,\bfj}\in \mathcal C_2.$
\end{proof}

In order to show that the sum corresponding to $\mathcal C_3$ in (\ref{catsum}) is negligible for $N\to\infty$, we introduce the concept of isomorphic $Q^+$-graphs. 

\begin{defn} Two colored $Q^+$-graphs $Q_{\bfl,\bfi,\bfj}$ and $Q_{\bfl',\bfi',\bfj'}$ are {\it isomorphic}, or $Q_{\bfl,\bfi,\bfj}\sim Q_{\bfl',\bfi',\bfj'},$ if we can obtain $Q_{\bfl,\bfi,\bfj}$ from $Q_{\bfl',\bfi',\bfj'}$ by permuting the numbers on the lines $\bfi,\bfj^{(1)},...,\bfj^{(m)}$. In particular, $Q_{\bfl,\bfi,\bfj}\sim Q_{\bfl',\bfi',\bfj'}$ implies $\bfl=\bfl'$.
\end{defn}

\begin{lem}\label{C_3sum}
It holds that
\begin{align*}
E_3:=N^{-1} n^{-k}\sum_{Q\in \mathcal C_3}(ty)_{Q}\to 0
\end{align*}
for $N,n\to\infty$ with $N/n\to c\in(0,\infty).$
\end{lem}
\begin{proof}
Observe the identity
\begin{align*}
E_3=N^{-1} n^{-k}\sum_{Q_3}\sum_{Q\in [Q_3]}(ty)_{Q},
\end{align*}
where the first summation is taken for a representative system of pairwise not isomorphic graphs in category $\mathcal C_3$ and the second summation for all $Q^+$-graphs isomorphic to $Q_3$. It is sufficient to show that for arbitrary $Q_3\in\mathcal C_3$ we have
\begin{align*}
N^{-1} n^{-k}\sum_{Q\in [Q_3]}(ty)_{Q}\ \to\ 0.
\end{align*}
Glue coincident vertical edges of $Q_3$ into colorless down edges. Let every vertical edge that connects with the $\bfj^{(l)}$-line correspond to the matrix
\[Y(\mu)=\left\{(\mu-1)!!\right\}_{N\times [n(t_l-t_{l-1})]},\]
where $\mu$ denotes the coincidence multiplicity of the edge.

 Applying Theorem A 35. of \cite{BS} and Lemma \ref{c3prop} yields that there is a constant $C_k$ such that 
\begin{align*}
N^{-1} n^{-k}\sum_{Q\in [Q_3]}(ty)_{Q} \leq C_k N^{-1} n^{-k} N^{r+s} = O(N^{-1}),
\end{align*}
 and the proof is complete.
\end{proof}


\subsection{Limit of the Expected $k$-th Spectral Moment}


In this subsection we derive the limit of the first sum in (\ref{catsum}). 
For a colored $Q^+$-graph $Q\in\mathcal C_1$, the expectation factor $\E\left[y^{(l_1)}_{i_2 j_1}y^{(l_1)}_{i_3 j_1}\cdots y^{(l_k)}_{i_{3k-1} j_k} y^{(l_k)}_{i_{3k} j_k}\right]$ of $(ty)_{Q}$ equals $1$. Therefore, 
\[N^{-1} n^{-k}\sum_{Q\in \mathcal C_1}(ty)_{Q}=N^{-1} n^{-k}\sum_{Q\in \mathcal C_1}(t)_{H(Q)}\]
depends on the heads of the graphs only.
Using the notations introduced in the last subsection, there are
\[\prod_{l=1}^m  [n(t_l-t_{l-1})]! \bigg/ \prod_{l=1}^m ([n(t_l-t_{l-1})]-s_l)!\] colored $Q^+$-graphs with the same head as $Q.$
Every graph $Q_1\in\mathcal C_1$ has $k$ noncoincident vertical edges and its pillar is a tree with $r+s$ vertices and $k$ edges where $s=s_1+...+s_m$. Consequently, we have $k=r+s-1$. 
 Therefore, it holds that
\begin{align}\label{C1sum}
N^{-1} n^{-k}\sum_{Q\in \mathcal C_1}(t)_{H(Q)}	
&= N^{-1} \sum_{H(Q)\in H(\mathcal C_1)} n^{-r+1} (t)_{H(Q)} \prod_{l=1}^m (t_l-t_{l-1})^{s_l}+o(1)
\end{align} 
where $H(\mathcal C_1)$ denotes the set of colored heads for graphs in $\mathcal C_1$.
\begin{percom}
From \[\prod_{l=1}^m\frac{[n(t_l-t_{l-1})]^{s_l}}{n^{s_l}}\to\prod_{l=1}^m (t_l-t_{l-1})^{s_l}\] we have 
\[\prod_{l=1}^m[n(t_l-t_{l-1})]^{s_l}-\prod_{l=1}^m(n(t_l-t_{l-1}))^{s_l}=o(n^s).\]
See also master thesis.
\end{percom}
We first derive the limit for this term if the summation runs for a class of similar heads.

\begin{figure}[b]
 \begin{tikzpicture}[>=latex,scale=0.65]
 \def\horizontaledge{[->] parabola bend +(1,0.2) +(2,0)}
 
  \draw (0,0)node[left]{\bf i}--+(10,0) ;
  
 \draw [green!70!black](1,0) \horizontaledge;
 \draw [green!70!black](7,0) \horizontaledge;
 \draw [red](3,0) \horizontaledge;
 \draw[green!70!black,->] (5,0) parabola bend +(-1,0.5) ++(-2,0);
 \draw[green!70!black,->] (3,0) parabola bend +(-1,0.5) ++(-2,0);
 \draw[red,->] (9,0) parabola bend +(-1,0.5) ++(-2,0);

   \draw (1,0) node [below] {\tiny $i_1=i_{10}$};
   \draw (5,0) node [below] {\tiny $i_7$};
   \draw (7,0) node [below] {\tiny $i_3=i_5$};
   \draw (9,0) node [below] {\tiny $i_4$};
 

  \draw (12,0)node[left]{\bf i}--+(10,0) ;
  
 \draw [green!70!black](13,0) \horizontaledge;
 \draw [green!70!black](17,0) \horizontaledge;
 \draw [red](15,0) \horizontaledge;
 \draw[green!70!black,->] (19,0) parabola bend +(-3,0.75) ++(-6,0);
 \draw[green!70!black,->] (21,0) ..controls +(120:1) and +(60:1).. +(0,0);
 \draw[red,->] (21,0) ..controls +(30:2) and +(150:2).. +(0,0);

   \draw (13,0) node [below] {\tiny $i_1=i_{10}$};
   \draw (15,0) node [below] {\tiny $i_2=i_{6}$};
   
   \draw (17,0) node [below] {\tiny $i_7$};
   \draw (19,0) node [below] {\tiny $i_8=i_9$};
   \draw (21,0) node [below] {\tiny $i_3=i_4=i_5$};

\end{tikzpicture}
\caption{Two similar heads.}
\label{simheads}
\end{figure}
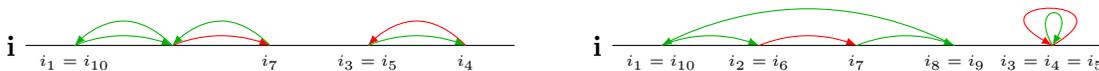

\begin{defn}\label{similar}
A $Q^+$-graph $Q$ induces a partition of the set $\{1,...,3k\}$, where $a$ and $b$ are in the same partition set if and only if $i_a$ and $i_b$ are connected in $H(Q)$. 
Let $Q$ and $Q'$ be colored $Q^+$-graphs with the same coloring vector. The heads $H(Q)$ and $ H(Q')$ are {\it similar} (sometimes we also say $Q$ and $Q'$ are similar) if they induce the same partition. The equivalence class of heads similar to $H(Q)$ will be denoted by $[[H(Q)]]$. See Figure \ref{simheads} for an example.
\end{defn}
At this point it is convenient to introduce the notion of component coloring multi-indices (CCMIs).
For a head of a colored $Q^+$-graph we denote the connected components by $H_1,...,H_r$ and their sizes (i.e. the number of edges they contain) by $2\nu_1,...,2\nu_r$.
For some component $H_a$ of the head, the CCMI $\bfl^{(a)}=(l^{(a)}_1,...,l^{(a)}_{\nu_a})\in\{1,...,m\}^{\nu_a}$ is defined in the following way.
We obtain a natural order for the edges of the $Q^+$-graph by the order of indices in (\ref{tydef}), i.e. the first edge connects $i_1$ and $i_2$, the second $i_2$ and $j_1$ and so on.
We set $l^{(a)}_b=l$ where $l$ is the color of the $b$-th up edge that connects to $H_a.$
\begin{remark}
Note that for a given $Q^+$-graph $Q$ the multi-index $(\bfl^{(1)},...,\bfl^{(r)})$ is not uniquely determined since it depends on the labeling of the head components $H_1,...,H_r$. We follow the convention that $H_1$ contains the index $i_1$. The labeling of the components $H_2,...,H_r$, however, remains arbitrary, for reasons that will be explained later.
\end{remark}
Let us now analyze the summation of $(t)_{H(Q)}$ for a similarity class $[[H(Q)]]$.

\begin{lem}\label{inclexclprinc}
Introducing the notation \[(TT^*)_{ \bfl}=\prod_{i=1}^{q} T_{l_i}T^*_{l_i}\]
for $\bfl=(l_1,...,l_q)\in \{1,...,m\}^q$, we have for a colored $Q^+$ graphs $Q_1\in\mathcal C_1$ with CCMIs $\bfl^{(1)},...,\bfl^{(r)}$ 
\[ \hspace{-1.5ex}\sum_{H(Q)\in [[H(Q_1)]]}\hspace{-1.5ex}(t)_{H(Q)}= \prod_{a=1}^r (\tr(TT^*)_{\bfl^{(a)}}+O(1)).\]
\end{lem}
\begin{proof}
We write
\begin{align}\label{cyclesummation}
\sum_{H(Q)\in [[H(Q_1)]]}\hspace{-1.5ex}(t)_{H(Q)}
&=\sideset{}{'}\sum (t)_{H_1}\cdots(t)_{H_{r}},\nonumber\\
\end{align}
where $(t)_{H_1},...,(t)_{H_r}$ are products of entries of $T_{1},...,T_m,T_1^*,...,T_m^*$ associated with the distinct cycles 
$H_1,...,H_r$ of the head in the same manner as colored $Q^+$-graphs are associated to products of the form (\ref{tydef}). The summation $\Sigma'$ runs for all indices of these entries over $\{1,...,N\}$, with the restriction that entries of different cycles have distinct indices.

Consider first $\sum (t)_{H_a}$ for some $a\in\{1,...,r\}$. By the definition of the CCMI $\bfl^{(a)}$ and recalling that the pillar of $Q$ is a tree it is not difficult to verify that
\[\sum (t)_{H_a}=\sum_{\bfi\in \{1,...,N\}^{2\nu_a}} t^{*(l^{(a)}_1)}_{i_1 i_2}t^{(l^{(a)}_2)}_{i_2 i_3}t^{*(l^{(a)}_2)}_{i_3 i_4}\dots t^{*(l^{(a)}_{\nu_a})}_{i_{2\nu_a-1} i_{2\nu_a}} t^{(l^{(a)}_1)}_{i_{2\nu_a} i_1}=\tr(TT^*)_{\bfl^{(a)}}.\]
Then, applying the inclusion-exclusion principle, and recalling $\|T_l\|_{\text{op}}\leq \tau_0$ for all $l$, yields 
\[\sum_{\bfi\in \{1,...,N\}^{2\nu_a}\atop \{i_1,...,i_{2\nu_a}\}\cap M=\emptyset} t^{*(l^{(a)}_1)}_{i_1 i_2}t^{(l^{(a)}_2)}_{i_2 i_3}\dots t^{(l^{(a)}_{1})}_{i_{2\nu_a} i_1}=\tr(TT^*)_{\bfl^{(a)}}+O(1),\]
for any finite set $M\subset \Z_+.$
The statement follows now by induction over the distinct cycles of the head.
\end{proof}
Applying this Lemma we find
\begin{align}\label{simlim}
 N^{-1} \sum_{H(Q)\in [[H(Q_1)]]} n^{-r+1} (t)_{H(Q)} \prod_{l=1}^m (t_l-t_{l-1})^{s_i}
&\quad\longrightarrow\quad  c^{r-1}\prod_{l=1}^m (t_l-t_{l-1})^{s_i} \prod_{a=1}^r M^{\nu_a}_{\bfl^{(a)}}.
\end{align}

Now, in order to derive the limit of $\E[m_k(F^{[X]^N_n})],$ it is sufficient to determine, for given $r, \nu_1,...,\nu_r,s_1,...,s_m$ and $\bfl^{(1)},...,\bfl^{(r)}$, the number of similarity classes with this specific parameters.

\begin{defn}
Two components $H_a$ and $H_b$ of the head of a colored $Q^+$-graph in $\mathcal C_1$ are {\it vertically connected} if there is a down edge starting at some vertex in $H_a$, which is followed by an up edge that ends at some vertex in $H_b.$ Note that if $H_a$ and $H_b$ are vertically connected, then there is exactly one down edge leaving $H_a$ that is followed by an up edge connecting to $H_b$ and exactly one down edge leaving $H_b$ that is followed by an up edge connecting to $H_a.$ These four edges form two pairs of coincident edges and are of the same color. Therefore, we may understand the vertical connections as colored as well.
\end{defn}

\begin{defn}
For a colored $Q^+$-graph $Q\in\mathcal C_1$ with components $H_1,...,H_r$ we define the connectivity tree $G_Q$ to be the graph with vertex set $\{H_1,...,H_r\}$ where $(H_a,H_b)$ is an edge in $G_Q$ if and only if $H_a$ and $H_b$ are vertically connected in $Q$.
\end{defn}
Note that due to the arbitrary labeling of $H_2,...,H_r$ most $Q^+$-graphs have more than one possible connectivity tree.

\begin{lem}\label{count}
It holds that 
\begin{align*}
&\E[m_k(F^{\wXpow})]\longrightarrow m_k
\end{align*}
where
\begin{align*}
m_k=\sum_{r=1}^{k}c^{r-1}\sum_{\nu_1+...+\nu_r=k}\ \sum_{\bfl'\in\{1,...,m\}^{k}}\ 
 c_{r,\nu,\bfl'}\prod_{a=1}^r M^{\nu_a}_{\bfl^{(a)}}\prod_{l=1}^m (t_l-t_{l-1})^{s_{l,\nu,\bfl'}}.
\end{align*}
Here, $\bfl'=(\bfl^{(1)},...,\bfl^{(r)})$ where $\bfl^{(a)}$ has length $\nu_a$. For the definition of $s_{l,\nu,\bfl'}$ see Theorem \ref{th1}, for the definition of $ c_{r,\nu,\bfl'}$ see section \ref{cdef}.
\end{lem} 

\begin{proof}
Recalling (\ref{catsum}) and Lemma \ref{C_3sum} it is sufficient to derive that
\[N^{-1}n^{-k}\sum_{Q\in \mathcal C_1}(ty)_Q\longrightarrow m_k.\]
Thus, by virtue of (\ref{C1sum}) and (\ref{simlim}) there are two things left to show:
\begin{description}
\item{ (1)} For a $Q^+$-graph $Q$ with CCMIs $\bfl^{(1)},...,\bfl^{(r)},$ the number of noncoincident vertices on the $\bfj^{(l)}$-line is $s_l=\sum_{a=1}^r n_l^{(a)}.$
\item{(2)} There are $c_{r,\nu,\bfl'}$ similarity classes of $Q^+$-graphs with CCMIs $\bfl^{(1)},...,\bfl^{(r)}$.
\end{description}
For (1) note that every vertex on the $\bfj^{(l)}$-line has either degree 2 or 4 and its degree is 4 if and only if it lies on a vertical connection of color $l$. Therefore, $s_l$ is the number of up edges colored in $l$ minus the number of vertical connections of color $l.$\\
The number of up edges colored in $l$ is the number of $l$-s in the CCMIs $\bfl^{(1)},...,\bfl^{(m)}$. Let $H_a$ and $H_b$ be two vertically connected components where in the connectivity tree $G_Q$ $H_a$ lies on the path from $H_b$ to $H_1$. Then, the color of the vertical connection $(H_a,H_b)$ is $l^{(b)}_1$. Therefore, the entries $l^{(b)}_1$ for $b>1$ correspond one to one to the colors of the vertical connections of $Q$.
This proves claim (1).

For $(2)$ we first show that there are
\[\prod_{l=1}^m \prod_{a=1}^{r}\frac{n_l^{(a)}!}{(n_l^{(a)}-n_l^{(a),G})!} 1_{\{n_l^{(a),G}\leq n_l^{(a)}\}}\]
similarity classes of $Q^+$-graphs with connectivity tree $G$ and CCMIs $\bfl^{(1)},...,\bfl^{(r)}$.
Within a component $H_a$ a vertical connection $(H_a,H_b)$ is at a certain position $p\in\{1,...,\nu_a\}$, meaning that the $p$-th down edge leaving $H_a$ is followed by an up edge connecting to $H_b$. It is straightforward to verify that two $Q^+$-graphs with the same connectivity tree $G$ and the same CCMIs $\bfl^{(1)},...,\bfl^{(r)}$ are similar if and only if within all components all vertical connections are at the same positions. 
 
Consider component $H_1$, and let $H_{a_1},...,H_{a_p}$ be the components adjacent to it in $G$. A $Q^+$-graph $Q$ with connectivity tree $G$ contains the corresponding vertical connections $(H_1,H_{a_1}),...,(H_1,H_{a_p})$, $n_l^{(1),G}$ of which are colored in $l$. Since $H_1$ has $n_l^{(1)}$ leaving down edges of color $l$ we have $\prod_{l=1}^m\frac{n_l^{(1)}!}{(n_l^{(1)}-n_l^{(1),G})!} 1_{\{n_l^{(1),G}\leq n_l^{(1)}\}}$ possibilities of positioning the vertical connections among the vertical edges leaving $H_1$. Now turn to some component $H_a\neq H_1$. There is one component $H_{a_0}$ vertically connected to $H_a$ that lies on the path from $H_a$ to $H_1$ in $G$.
By construction, the vertical connection $(H_a,H_{a_0})$ is at position $\nu_a$ within $H_a$ and it is colored in $l^{(a)}_1.$ 
For distributing all other vertical connections at $H_a$ on their possible positions within $H_a$, we are left with $\prod_{l=1}^m\frac{n_l^{(a)}!}{(n_l^{(a)}-n_l^{(a),G})!} 1_{\{n_l^{(a),G}\leq n_l^{(a)}\}}$ possibilities. 
This leaves us, overall, with
\[\prod_{l=1}^m \prod_{a=1}^{r}\frac{n_l^{(a)}!}{(n_l^{(a)}-n_l^{(a),G})!} 1_{\{n_l^{(a),G}\leq n_l^{(a)}\}}\]
possibilities for distributing all vertical connections of all components on their possible positions. 

Most similarity classes have more than one possible connectivity tree and CCMIs since the components $H_2,...,H_r$ are arbitrarily labeled. By definition of the set $S_{\bfl',G}$, introduced in section \ref{cdef}, a $Q^+$-graph $Q$ has $|S_{\bfl',G}|$ possible connectivity trees and CCMIs where $G$ is one possible connectivity tree for $Q$. This proves (2).
\end{proof}

\begin{remark}
The arbitrary labeling of the components $H_2,...,H_r$ is necessary in order to apply the combinatorical arguments of the proof above. If we, for example, label the components in order of appearance with respect to the natural order of edges, we impose subtle restrictions on the CCMIs, leading to more involved expressions.
\end{remark}

In the next subsection we complete the proof of Theorem \ref{th1}.


\subsection{Convergence of $m_k(F^{\XN})$}
\label{m_kcon}

The following Lemma ensures the a.s. convergence of $m_k(F^{\wXpow})$. The proof relies on corresponding results for constant $f$. For more details we refer to \cite[Theorem 4.1]{BS}.
\begin{lem}\label{summable} It holds that
\begin{align*}
\E[m_k(F^{\wXpow})-\E m_k(F^{\wXpow})]^4 = O(N^{-2}). 
\end{align*}
\end{lem}

\begin{proof}
For $a=1,...,4$, given multi-indices $\bfl_a=(l_1^{(a)},...,l_k^{(a)})\in\{1,...,m\}^k,$ $\bfi_a\in\{1,...,N\}^{3k}$ and $\bfj_a=(j^{(a)}_1,...,j^{(a)}_k)$ with $j^{(a)}_p\in \{1,...,[n(t_{l^{(a)}_p}-t_{l^{(a)}_p-1})]\}$, we denote by $Q_a$ the corresponding colored $Q^+$-graph $Q_{\bfl_a,\bfi_a,\bfj_a}$. Then, we have

\begin{align}\label{power4sum}
\E[m_k(F^{\wXpow})-\E m_k(F^{\wXpow})]^4
&=\E\left[\frac 1 {N n^k} \sum_{\bfl, \bfi, \bfj} (ty)_{Q_{\bfl,\bfi,\bfj}}-\E\left[\frac 1 {N n^k}\sum_{\bfl,\bfi,\bfj} (ty)_{Q_{\bfl,\bfi,\bfj}}\right]\right]^4\nonumber\\
&=N^{-4}n^{-4k}\sum_{\bfl_1,...,\bfj_4} \E\left[ \prod_{a=1}^4 \left((ty)_{Q_a}-\E [(ty)_{Q_a}]\right)\right].
\end{align}
If, for some $a$, all vertical edges of $Q_a$ do not coincide with vertical edges of one of the other graphs, we obtain
\[\E\left[ \prod_{a=1}^4 \left((ty)_{Q_a}-\E [(ty)_{Q_a}]\right)\right]=0,\]
from independence. Thus, $Q=\cup Q_a$ consists of either one or two connected components. By expanding (\ref{power4sum}) we have
\begin{align}\label{power4sum2}
&\hspace {-5ex}\E[m_k(F^{\wXpow})-\E m_k(F^{\wXpow})]^4\nonumber\\
&=N^{-4} n^{-4k}\sum_{\bfi_1,...,\bfl_4} \left( \E\left[\prod_{a=1}^4 (ty)_{Q_a}\right]\pm...+\prod_{a=1}^4\E [(ty)_{Q_a}]\right)\nonumber
\end{align}
Applying Theorem A.35. of \cite{BS}, in a similar way as in the proof of Lemma \ref{C_3sum}, for each of the 16 summands within the  brackets separately, shows that this sum is $O(N^{-2}).$ 
\end{proof}

Now, combining Lemma \ref{count} and Lemma \ref{summable} we have, by virtue of the Borel-Cantelli Lemma and (\ref{error}),
\[m_k(F^{[X]_n^N})\overset{\text{\tiny a.s.}}{\longrightarrow} m_k,\]
for all $k$, where $m_k$ is defined as in Theorem \ref{th1}. Therefore, if the sequence $(m_k)$ satisfies Carleman's condition, applying Theorem \ref{MCT} completes the proof of Theorem \ref{th1}. 

\begin{lem}\label{Carleman}
The sequence of limiting spectral moments $m_k$ satisfies
\[\sum_{k=0}^\infty (m_{2k})^{-\frac{1}{2k}}=\infty.\]
\end{lem}
\begin{proof}
Consider the matrices $S_l=\frac 1 {n} T_lY_lY_l^*T_l^*$ for $l=1,...,m.$
The spectral distribution of $\frac 1{n(t_l-t_{l-1})} Y_lY_l^*$ is known to converge to the Mar\v{c}enko-Pastur law $p_{y_l}(x)$ with support $[(1-\sqrt{y_l})^2,(1+\sqrt{y_l})^2]$ for $l=1,...,m$, where $y_l=c(t_l-t_{l-1})^{-1}$. Thus, we have
\[\lim_{N\to\infty}\frac 1 N \tr(S_l^k)\leq \lim_{N\to\infty}(t_l-t_{l-1})^k \|T_l\|_{\text{op}}^{2k}\bigg\|\frac 1 {n(t_l-t_{l-1})} Y_l^*Y_l\bigg\|_{\text{op}}^k\leq\tau_0^{2k}\left(1+\sqrt {\sideset{}{_l}{\max}(y_l)}\right)^{2k},\]
for $l=1,...,m.$
Therefore, the result follows from
\begin{align*}
m_k(F^{[X]^N_n})\leq \frac 1 N m^{k-1}\left(\tr(S_1^k)+\dots+\tr(S_m^k)\right),
\end{align*}
which holds due to the convexity of the function $A\mapsto \tr(A)^k$, see for example \cite[Theorem 2.10]{C}.
\end{proof}

\textit{Acknowledgment.} We would like to thank Carlos Vargas Obieta and Steen Thorbj{\o}rnsen for a fruitful discussion.

\end{document}